\documentclass[12pt,a4paper]{amsart}
\usepackage[utf8]{inputenc}
\usepackage{amsmath}
\usepackage{amsfonts}
\usepackage[colorlinks,linktocpage=true]{hyperref}
\hypersetup{colorlinks,linkcolor={blue},citecolor={blue},urlcolor={black}}
\usepackage{cite}
\usepackage{array,booktabs}
\usepackage{amssymb}
\usepackage{enumerate}
\usepackage[colorinlistoftodos]{todonotes}
\usepackage{esint}
\theoremstyle{plain}
\newtheorem{theorem}{Theorem}[section]
\newtheorem{proposition}[section]{Proposition}
\newtheorem{corollary}[theorem]{Corollary}
\newtheorem{lemma}[theorem]{Lemma}
\newtheorem{question}[theorem]{Question}
\theoremstyle{definition}
\newtheorem{definition}{Definition}
\theoremstyle{remark}

\newtheorem*{acknowledgements}{Acknowledgements}

\newcommand{\diam}{\operatorname{diam}}

\DeclareMathOperator*{\dist}{dist}
\DeclareMathOperator*{\lip}{Lip}
\DeclareMathOperator*{\acl}{ACL}
\DeclareMathOperator*{\md}{MD}
\DeclareMathOperator*{\modu}{Mod_n}
\DeclareMathOperator*{\modup}{Mod_{d,n}}
\DeclareMathOperator*{\modue}{Mod_{2}}
\def\C{\mathbb C}
\def\R{\mathbb R}

\def\N{\mathbb N}

\usepackage[showonlyrefs]{mathtools}
\mathtoolsset{showonlyrefs}

\numberwithin{equation}{section}
\author{Lauri Hitruhin}
\address{University of Helsinki, Department of Mathematics and Statistics, P.O. Box 68, FIN-00014 University of Helsinki, Finland}
\email{lauri.hitruhin@helsinki.fi}
\author{Athanasios Tsantaris}
\address{University of Helsinki, Department of Mathematics and Statistics, P.O. Box 68, FIN-00014 University of Helsinki, Finland}
\email{athanasios.tsantaris@helsinki.fi}
\makeatletter
\@namedef{subjclassname@2020}{%
	\textup{2020} Mathematics Subject Classification}

\makeatother

\title[Quasiconformal curves]{Quasiconformal curves and quasiconformal maps in metric spaces}

\begin{document}
	\let\thefootnote\relax\footnotetext{This work was supported by the Academy of Finland projects \#332671 and SA-1346562 }
	\maketitle
	\begin{abstract}
		In this paper we study quasiconformal curves which are a special case of quasiregular curves. Namely embeddings $\Omega\to\R^m$ from some domain $\Omega\subset\R^n$ to $\R^m$, where $n\leq m$, which belong in a suitable Sobolev class and satisfy a certain distortion inequality for some smooth, closed and non-vanishing $n$-form in $\R^m$. These mappings can be seen as quasiconformal mappings between $\Omega$ and $f(\Omega)$. We prove that a quasiconformal curve always satisfies the analytic definition of quasiconformal mappings and the lower half of the modulus inequality. Moreover, we give a sufficient condition for a quasiconformal curve to satisfy the metric definition of quasiconformal mappings. We also show that a quasiconformal map from $\Omega$ to $f(\Omega)\subset \R^m$ is a quasiconformal $\omega$ curve for some form $\omega$ under suitable assumptions.  Finally, we show the same is true when we equip the target space $f(\Omega)$ with its intrinsic metric instead of the Euclidean one.
	\end{abstract}
\section{Introduction}
Quasiconformal mappings between Riemannian manifolds are a generalization of conformal mappings, i.e. mappings whose pointwise derivative is the scalar multiple of an orthogonal transformation.  We refer the reader to the books \cite{Gehring2017,Astala2009, Vaeisaelae1971} for thorough introductions to the theory of quasiconformal mappings.

There are three ways to define quasiconformal mappings. The so called analytic, geometric and metric definitions. Each of these definitions can  be given in the setting of metric spaces. Let $(X, d_X,\mu)$ and $(Y,d_Y,\nu)$ be metric measure spaces.
\begin{definition}{(Metric definition)}
	A homeomorphism $f:X\to Y$ is called $H$-quasiconformal if there exists a constant $H<\infty$ such that \begin{equation}\label{eq000}H_f(x):=\limsup_{r\to 0}\frac{L_f(x,r)}{\ell_f(x,r)}\leq H,\end{equation} for all $x\in X$, where $$L_f(x,r)=\sup_{d_X(x,y)\leq r}d_Y(f(x),f(y)) \ \text{and}\  \ell_f(x,r)=\inf_{d_X(x,y)\geq r}d_Y(f(x),f(y)).$$
	
\end{definition}  
 The analytic definition is formulated via the Newton--Sobolev spaces $N^{1,n}_{loc}(X,Y)$, which were introduced in \cite{Shanmugalingam2000} and are one of the many ways to define Sobolev spaces in the setting of metric spaces, see \cite{Hajlasz2003} or \cite{Heinonen2015} for example. We also require the notions of the Lipschitz constant \[\lip f(x)=\limsup_{y\to x}\frac{d_Y(f(x),f(y))}{d_X(x,y)},\] and that of the volume derivative of $f$, \[\mu_f=\limsup_{r\to0}\frac{\nu(f(B(x,r)))}{\mu(B(x,r))}.\]  In \cite{Heinonen2001} it was proved that if the spaces $X$ and $Y$ have locally $n$-bounded geometry then the metric definition is equivalent to both the analytic and the geometric definitions.
 \begin{definition}{(Analytic definition)}\label{analyticdef}
 	A homeomorphism $f:X\to Y$ is called $K$-quasiconformal (of index $n$) if $f\in N^{1,n}_{loc}(X,Y)$ and there exists a constant $K\geq 1$  such that \begin{equation}\label{eq01}
 		\lip f(x)^n\leq K \mu_f,
 	\end{equation} almost everywhere in $X$.
 \end{definition}
 
\begin{definition}{(Geometric definition)}
	A homeomorphism $f:X\to Y$ is called $K'$-quasiconformal (of index $n$) if the inequalities
	\begin{equation}\label{eq02}\frac{1}{K'}\modu \Gamma\leq \modu f(\Gamma)\leq K' \modu \Gamma \end{equation}
	hold for every family of paths $\Gamma$ in $X$.
\end{definition}
  Here  $\modu$ denotes the $n$-modulus of the path family $\Gamma$ with respect to the measure of corresponding metric space, see Section \ref{prelim}. Notice that in the above definitions we do not assume that the metric spaces $X$ and $Y$ have $n$-bounded geometry or even that they are Ahlfors $n$-regular. 
  
It is also interesting to point out that, as was proved in \cite{Heinonen1998}, a metric quasiconformal map $f:X\to Y$ between two $n$-Ahlfors regular metric spaces $X$ and $Y$, with $X$ being Loewner and $Y$ being linearly locally connected, is also locally quasisymmetric (see Section \ref{prelim} for terminology). Thus local quasisymmetry is equivalent to quasiconformality with any of the three definitions when the metric spaces $X$ and $Y$ have locally $n$-bounded geometry, see \cite[Theorem 9.8]{Heinonen2001}.

Recently, Pankka in \cite{Pankka2020} generalized the notion of quasiregular maps in the setting where the domain of definition and the target no longer have the same dimension. Let $\Omega^n(\R^m)$  denote the space of smooth differential $n$-forms in $\R^m$, where $n\leq m$. In what follows we will assume that $\omega\in \Omega^n(\R^m)$ is a smooth, non-vanishing and closed differential $n$-form. We call such forms $n$-\textit{volume forms}.

\begin{definition}[Quasiregular $\omega$-curve]\label{def0}
	A mapping $f\in W^{1,n}_{loc}(\Omega,\R^m)$, where $\Omega$ is a domain in $\R^n$ and $n\leq m$ is called a $K$-quasiregular $\omega$-curve if  there exists  a constant $K\geq 1$  such that the  following inequality is satisfied 
	\begin{equation}\label{eq03}
		(||\omega||\circ f) |Df(x)|^n\leq K (\star f^*\omega),
	\end{equation}
	almost everywhere on $\Omega$, where $|Df(x)|$ denotes the supremum norm of the differential of $f$.
\end{definition} 
Here $\star f^*\omega$ is the Hodge star of the $n$-volume form $f^*\omega$ (the pullback of $\omega$), i.e. the function that satisfies $(\star f^*\omega) dx_1\wedge\dots\wedge dx_n=f^*\omega$. The function $||\omega||\colon \R^m\to[0,\infty)$ is the pointwise comass norm of $\omega$ defined as  \begin{equation}\label{eq04}
	||\omega||(p)=\sup\{|\omega_p(v_1,\dots,v_n)|\colon \ v_1,\dots,v_2\in \R^m, \ |v_i|\leq 1\}
\end{equation}
for each $p\in\R^m$.

In fact the theory developed also includes the analogues of  finite distortion mappings, see \cite{Pankka2020, Heikkilae2022, Onninen2021,Hitruhin2022, Heikkilae2023}, but here we are interested in quasiconformal maps only. 

\begin{definition}[Quasiconformal $\omega$-curve]\label{def1}
	A mapping $f:\Omega\to\R^m$, where $\Omega$ is a domain in $\R^n$ and $n\leq m$, is called a $K$-quasiconformal $\omega$-curve if it is a $K$-quasiregular $\omega$-curve and a topological embedding.
\end{definition} 

 The above definition can be seen as the  analogue of Definition \ref{analyticdef} in the case where the target has larger dimension than the domain. Notice that a quasiconformal $\omega$-curve $f:\Omega\to \R^m$ can also be seen as a map $f:\Omega\to f(\Omega)$. We can now equip $f(\Omega)$ with the euclidean metric or with the intrinsic metric.  Hence, we obtain a map between metric spaces of the same topological dimension. 
 
  The question that arises naturally now is whether quasiconformal $\omega$-curves satisfy the three definitions of quasiconformal mappings between metric spaces and if they are locally quasisymmetric. We note here that the three definitions of quasiconformal mappings are known to be equivalent only when the metric spaces involved have locally $n$-bounded geometry and that the image of a quasiconformal curve $f$ does not have this property a priori. Thus it is not clear which definitions of quasiconformality, if any, will a quasiconformal curve satisfy. The main goal of this paper is to explore the interplay between  all these notions.    
 
 We will restrict ourselves to constant coefficient forms. An $n$-volume form in $\R^m$,  $\omega=\sum_I \phi_I(x)dx_I$, where the sum is over all multi-indices $I=(i_1,\dots, i_n)$, $1\leq i_1< \dots< i_n\leq m$,  is called a \textit{constant coefficient form} when the functions $\phi_I$ are constant. Here $dx_I$ denotes the $n$-covector $dx_{i_1}\wedge\dots\wedge dx_{i_n}$. Notice that for constant coefficient forms the comass norm is constant.

As we have already mentioned a quasiconformal curve can be seen as a map between the metric spaces $\Omega$ and $f(\Omega)$ when equipped with a suitable metric. Thus our first goal is to settle whether  $f(\Omega)$ is a metric space of locally bounded $n$-geometry. Since $\Omega$ is always a space of locally bounded $n$-geometry this would mean that the theorem of Heinonen, Koskela, Shanmuganlingam and Tyson from \cite{Heinonen2001} on the equivalence of the definitions of quasiconformality can be used. Unfortunately, as our example in Section \ref{section3} shows this is not the case in general. However, Williams in \cite{Williams2012} has proven that the analytic definition is equivalent the lower half of the modulus inequality in a much more general setting.

Our first result relates the definition of quasiconformal curves to the analytic and geometric definitions of quasiconformality by using the aforementioned result of Williams. 
\begin{theorem}\label{thmmain}
	Let $f:\Omega\to f(\Omega)\subset \R^m$ be an embedding for which $f(\Omega)$ has  locally finite  Hausdorff $n$-measure of $\R^m$. If there exist constants $C>0$, $K\geq 1$ and an $n$-volume form with constant coefficients $\omega$  such that $\mu_f\leq C \star f^*\omega$, then the following are equivalent:
	\begin{enumerate}
	\item $f$ is a quasiconformal  $\omega$-curve,
	\item $f$ is in $N^{1,n}_{loc}(\Omega,\R^m)$ and $\lip f(x)^n\leq K \mu_f$ for almost every $x\in \Omega$, and
	\item $\modu \Gamma\leq K\modu f(\Gamma)$, for any path family $\Gamma$ in $\Omega$.
	\end{enumerate}
\end{theorem}

The assumption that $f(\Omega)$ has locally finite Hausdorff measure is natural. Indeed, since  $f$ is a quasiconformal curve that is always the case due to the fact that the area formula holds, see section \ref{prelim}. 

Next we examine what happens with the metric definition of quasiconformality. To state our results we need to introduce some notation. Suppose that $f=(f_1,\dots,f_m):\Omega\to\R^m$ is a quasiconformal $\omega$-curve. Let $y=(y_1,\dots,y_m)\in f(\Omega)$, we  define $f_I=(f_{i_1},\dots,f_{i_n})$ and $y_I=(y_{i_1},\dots,y_{i_n})$ where $I=(i_1,\dots,i_n)$, and set \[N(f_I,y,A)=\# (f_I^{-1}(y_I)\cap A),\] for each  $A\subset \R^n$. We also let \[N(f_I)=\sup_{y\in f(\Omega)}N(f_I,y,\Omega).\]
We will say that $f\Omega\to \R^m$ is \textit{projection finite} if $\max_I N(f_I)<\infty$, where the maximum is taken over all multi indices.
We are now ready to state our result.
\begin{theorem}\label{thmmetric}
	Let $f:\Omega\to f(\Omega)\subset \R^m$ be a projection finite embedding for which  $f(\Omega)$ is an Ahlfors $n$-regular metric space with the euclidean metric. If there exist a constant $C>0$ and an $n$-volume form with constant coefficients $\omega$  such that $\mu_f\leq C \star f^*\omega$    then the following are equivalent:
	\begin{enumerate}
		\item $f$ is a quasiconformal $\omega$-curve,
		\item  $H_f(x)\leq H$, for all $x\in \Omega$.
	\end{enumerate}
\end{theorem}

Next we discuss the relation of quasiconformal curves and quasisymmetries. Our next result shows quasiconformal curves are not necessarily local quasisymmetries without additional assumptions.

\begin{theorem}\label{thmquasisym}
	There exists a quasiconformal $\omega$-curve $f:\R^2\to\R^3$, where $\omega=dx_1\wedge dx_2$, such that $f$ is not locally quasisymmetric when $f(\R^2)$ is equipped with the euclidean metric.
\end{theorem}
By Theorem \ref{thmmetric} (see also Theorem \ref{thm1}) and  Theorem 9.8 of Heinonen et al. in \cite{Heinonen2001} we know that a quasiconformal curve such that $f(\Omega)$ is a metric space of bounded $n$-geometry and  $\max_I N(f_I)<\infty$ is always a local quasisymmetry. We do not know however if the condition of bounded $n$-geometry on $f(\Omega)$ is necessary.
\begin{question}
	Is a projection finite quasiconformal curve  a local quasisymmetry?
\end{question}

Notice that in all the above theorems we have always assumed that $f(\Omega)$ is equipped with the euclidean metric. It is also interesting to examine what happens when we equip $f(\Omega)$ with its intrinsic metric. By that we mean the metric \[d(x,y)=\inf_\gamma\ell (\gamma),\]where $\ell$ denotes the length of the path and the infimum is taken over all rectifiable paths connecting $x$ and $y$ in $f(\Omega)$. Note that $d$ is not necessarily a finite metric so the distance between two points might be infinite. But we will show that for the mappings we are interested in, that is always the case and $d$ defines a metric in the classical sense. We will denote by $\mathcal{H}^n_d$ the Hausdorff $n$-measure with respect to the metric $d$ and by $\mathcal{H}^n_e$ the Hausdorff $n$-measure with respect to the euclidean metric. Notice that by changing the metric on the target we obtain three new definitions for quasiconformality.

\begin{theorem}\label{thm5}
	Let $f:\Omega\to \R^m$ be a quasiconformal $\omega$-curve for a constant coefficient $n$-volume form $\omega$ in $\R^m$.  Then for any  measurable $A\subset \Omega$ we have $\mathcal{H}^n_d(f(A))=\mathcal{H}^n_e(f(A))$.
\end{theorem}

As a corollary we obtain the analogue of Theorem \ref{thmmain} when the target space is equipped with its intrinsic metric. The notations $N^{1,n}_{d,loc}$, $\lip_d f$, $\rho_{d,f}$, $\mu_{d,f}$ and $\modup$ correspond to the Newton--Sobolev spaces, Lipschitz constant, minimal weak upper gradient, volume derivative and modulus of a path family when we equip $f(\Omega)$ with its intrinsic metric and $\Omega$ with the euclidean metric. 
\begin{theorem}\label{thmpath}
	Let $f:\Omega\to f(\Omega)\subset \R^m$ be an embedding and $f:\Omega\to f(\Omega)$ is a homeomorphism when $f(\Omega)$ is equipped with its intrinsic metric. Assume that $f(\Omega)$ is  locally finite with the Hausdorff $n$-measure of $d$. If there exists constant $C>0$ and an $n$-volume form with constant coefficients $\omega$  such that $\mu_{d,f}\leq C \star f^*\omega$ then the following are equivalent:
	\begin{enumerate}
		\item $f$ is a quasiconformal $\omega$-curve,
		\item $f$ is in $N^{1,n}_{d,loc}(\Omega,f(\Omega))$ and $\lip_d f(x)^n\leq K \mu_{d,f}$ for almost every $x\in \Omega$, and
		\item $\modu \Gamma\leq K\modup f(\Gamma)$ for any path family in $\Omega$.
	\end{enumerate}
\end{theorem}

 	It interesting to ask whether a quasiconformal $\omega$ curve, $f:\Omega\to \R^m$ satisfies also the upper half of the modulus inequality. It is not hard to see that this amounts to the fact that $f^{-1}: f(\Omega)\to \Omega$ belongs to the right Sobolev space.
 	
 	\begin{question}
 		Let $f:\Omega\to \R^m$ be a quasiconformal $\omega$-curve for a constant coefficient $n$-volume form $\omega$ in $\R^m$. Does $f^{-1}:f(\Omega)\to \Omega$ belong to $N^{1,n}_{loc}(f(\Omega),\Omega)$?
 	\end{question}

	The question about the equivalence of definitions for quasiconformal curves also makes sense in the setting of quasiregular curves. However, the theory of quasiregular mappings between metric spaces has only been developed for branched coverings, see \cite{Guo2016}. Quasiregular curves are not necessarily discrete mappings. Iwaniec, Verchota and Vogel in \cite{Iwaniec2002a} construct a quasiregular curve $f:\C\to\C^2$ such that $f$ is constant on the lower half plane, see also \cite{Heikkilae2022} for a more general example.

\subsection{Structure of the paper}
After this introduction we collect some basic definitions and notation in Section \ref{prelim}. To prove Theorems \ref{thmmain} and \ref{thmmetric} we split them to two parts. The first part is proving that a  quasiconformal curve satisfies the definitions of quasiconformality and the second is the reverse direction. 

In Section \ref{section3} we prove the first half of Theorems \ref{thmmain} and \ref{thmmetric} namely that a quasiconformal curve satisfies the analytic, geometric and metric definitions. That is done in Theorems \ref{thm2} and \ref{thm1}. The first by showing that a quasiconformal curve always satisfies the analytic definition of quasiconformal maps in the setting of metric spaces and then invoking Williams' theorem, see Section \ref{prelim}. The second by first showing that a quasiconformal curve always has an image which is upper Ahlfors regular if we assume that $\max_I N(f_I)<\infty$ and then by adapting the classical proof for quasiconformal mappings to our setting.

In Section \ref{section4} we prove the second half of Theorems \ref{thmmain} and \ref{thmexample}, roughly speaking how to go from the analytic, geometric and metric definitions of quasiconformality to quasiconformal curves. 

In Section \ref{section5} we prove  Theorem \ref{thm5} and its corollary Theorem \ref{thmpath}. For that we require the notion of metric differentiability first introduced by Kirchheim in \cite{Kirchheim1994} and a change of variables formula for Sobolev functions with values in metric spaces  by Karmanova in \cite{Karmanova2007} (see Section \ref{prelim}) which allows one to show that the Radon-Nikodym derivatives of the two Hausdorff measures in Theorem \ref{thm5} are the same. 
\begin{acknowledgements}
	We would like to thank Pekka Pankka and Toni Ikonen for many useful discussions and comments on these topics.
\end{acknowledgements}
\section{Preliminaries}\label{prelim}
\subsection{Modulus of a path family}
An important concept that we are going to use throughout the paper is that of the modulus of a path family. A path $\gamma$ in a metric measure space $(X,d,\mu)$, is a continuous map $\gamma:[a,b]\to X$. The length function $s_\gamma:[a,b]\to\R$ of a path $\gamma$ is defined as \[s_\gamma(t):=\sup_{a\leq x_1\leq\dots\leq x_n\leq t} \sum_{i=1}^n|\gamma(x_{i})-\gamma(x_{i-1})|,\] where the supremum is over all partitions of $[a,t]$. The length of $\gamma$ is $l(\gamma)=s_\gamma(b)$. If $\gamma$ has finite length then we say that $\gamma$ is \textit{rectifiable}. Moreover, a path $\gamma$ is called \textit{locally rectifiable} if each closed subpath of $\gamma$ is rectifiable. If $\gamma$ is rectifiable there exists a unique path $\gamma^0:[0,l(\gamma)]\to X$ such that \[\gamma(t)=\gamma^0(s_\gamma(t)), \ \text{for each}\ t\in[a,b]\] The path $\gamma^0$ is called the arc-length parametrization of $\gamma$. If $\gamma$ is a rectifiable path and $\rho:X\to [0,\infty]$ a Borel function, the line integral of $\rho$ along $\gamma$ is defined by\[\int_\gamma \rho=\int_0^{l(\gamma)}\rho(\gamma^0(t))dt.\] Let $\Gamma$ be a family of paths in a metric  measure space $(X,d,\mu)$. A Borel function $\rho:X\to[0,\infty]$ is said to be admissible for $\Gamma$ if \[\int_\gamma \rho\geq 1\] for every rectifiable path $\gamma\in \Gamma$. The $p$-modulus of $\Gamma$ is defined as \[\modu{}_p(\Gamma)=\inf_\rho\int_X \rho ^p d\mu,\] where the infimum is taken over all admissible functions for $\Gamma$. By convention, we say that $\modu{}_p(\Gamma)=\infty$ if there are no admissible functions for $\Gamma$.

We say that a property holds for $p$-almost every path if the property fails only on a path family $\Gamma$ such that $\modu{}_p(\Gamma)=0$.

\subsection{Newton--Sobolev spaces and upper gradients}
We refer to \cite{Hajlasz2003} or \cite{Heinonen2015} for comprehensive introductions to the theory of Newton--Sobolev spaces. Here we recall some basic facts that we are going to need for our purposes.

We want to define the Newton--Sobolev spaces $N^{1,n}_{loc}(X,Y)$, where $X$ and $Y$ are metric measure spaces. A function $f:(X,d_X,\mu)\to(Y,d_Y,\nu)$, where $\mu$ is finite, is said to be in $L^p(X,Y)$ if $d_Y(f(x),z)\in L^p(X)$ for some, and therefore for all, $z\in Y$. A measurable function $f:X\to Y$  is said to be in  $\tilde{N}^{1,p}(X,Y)$ if $f\in L^p(X,Y)$ and if there exists a Borel function $\rho:X\to[0,\infty]$ so that $\rho\in L^p(X)$ and \begin{equation}\label{eq001}
	|f(\gamma(a))-f(\gamma(b))|\leq \int_\gamma \rho ds,
\end{equation} for $p$-almost every rectifiable path $\gamma:[a,b]\to X$. We call such a function $\rho$ a $p$-weak upper gradient for $f$. If equation \eqref{eq001} holds for all rectifiable curves we call $g$ an upper gradient for $f$. If a function $f$ has a $p$-weak upper gradient then it also has an upper gradient, see for example \cite[Lemma 6.2.2]{Heinonen2015}. 

 Notice that \eqref{eq001} is invariant under a change of parameter so we can always assume that $\gamma$ is parametrized by arc-length. A $p$-weak upper gradient $\rho$ of $f$ is called minimal if for every weak upper gradient $\rho'$ of $f$ we have that $\rho\leq \rho'$ $ \mu$ almost everywhere. If $f$ has a  weak upper gradient in $L^p_{loc}(X)$ then $f$ has a unique (up to measure zero) minimal $p$-weak  upper gradient, see \cite[Theorem 7.16]{Hajlasz2003}, which we will denote by $\rho_f$. We equip the vector space $\tilde{N}^{1,p}(X,Y)$ with the seminorm \[||f||_{\tilde{N}^{1,p}(X,Y)}=||f||_{L^p(X,Y)}+||\rho_f||_{L^p(X)}.\] The Newton--Sobolev space $N^{1,p}(X,Y)$ is now obtained by passing to equivalence classes of functions in $\tilde{N}^{1,p}(X,Y)$, where $f_1\sim f_2$ if and only if $||f_1-f_2||_{\tilde{N}^{1,p}(X,Y)}=0$. Next we define the local Newton--Sobolev spaces. Let $\tilde{N}^{1,p}_{loc}(X,Y)$ be the vector space of functions $f:X\to Y$ such that every point $x\in X$ has an open neighbourhood $U_x$ in $X$ such that $f\in \tilde{N}^{1,p}(U_x,Y)$. Two function $f_1$ and $f_2$ in $\tilde{N}_{loc}^{1,p}(X,Y)$ are said to be equivalent if every point $x\in X$ has an open neighbourhood $U_x$ such that $$||{f_1}_{|U_x}-{f_2}_{|U_x}||_{\tilde{N}^{1,p}(U_x,Y)}=0.$$  The local Newton--Sobolev ${N}_{loc}^{1,p}(X,Y)$ space is now defined to be the equivalence classes of functions in $\tilde{N}_{loc}^{1,p}(X,Y)$ under the preceding equivalence relation. Notice that to define the local Newton--Sobolev space we only require that $X$ is locally finite.

 The next lemma tells us what happens in the case of maps between euclidean spaces, see \cite[Corollary 7.15]{Hajlasz2003} and \cite[Remark 3.24]{Heinonen2001}. We remind here that we always assume that $n\leq m$.
 
 \begin{lemma}\label{minimalgradient}
 	For any function $f:\Omega\to\R^m$,  where $\Omega$ domain in $\R^n$ and $m\geq 1$, we have that $f\in N^{1,p}_{loc}(\Omega,\R^m)$ if and only if $f\in W^{1,p}_{loc}(\Omega,\R^m)$, $1\leq p<\infty$.	Moreover, any function $f\in W^{1,p}_{loc}(\Omega,\R^m)$ has $|Df|$ as the least $p$-weak upper gradient. 
\end{lemma}

We are also going to need the following  important lemmas, see for example \cite[Theorem 5.7]{Hajlasz2003} and \cite[Chapter II, Theorem 2.3]{Rickman}. Notice that in \cite{Rickman} the theorem is stated for $\acl^p$ mappings but it is well known that this is equivalent to continuous $W_{loc}^{1,p}$ mappings, see for example \cite[Chapter I, Proposition 1.2]{Rickman}.
\begin{lemma}\label{lemmaaprox}
	Let $u_k:X\to \R\cup\{-\infty,\infty\}$ be a sequence of Borel functions which converge to a Borel function $u:X\to \R\cup\{-\infty,\infty\}$ in $L^p(X)$. Then there exists a subsequence $u_{k_j}$ such that \[\int_\gamma |u_{k_j}-u|\to 0, \ \text{as}\ j\to\infty\] for $p$-almost every rectifiable path  $\gamma$.
\end{lemma}
\begin{lemma}\label{fugledelemma}
	Let $f:\Omega\to \R^m$ be a $W^{1,p}_{loc}(\Omega,\R^m)$ function and let $\Gamma$ be the family of all locally rectifiable paths in $\Omega$ which have a closed subpath on which $f$ is not absolutely continuous. Then $\modu{}_p(\Gamma)=0$.
\end{lemma}
Next we  show that for an embedding $f:\Omega\to \R^m$ in $W^{1,p}_{loc}(\Omega,\R^m)$, for some $p>n$ we have that   $f(\Omega)$ is actually a rectifiably connected metric space. It follows that  the intrinsic metric $d$ is always finite in $f(\Omega)$ when $f$ is a quasiconformal $\omega$-curve since quasiconformal curves have higher integrability, see Theorem \ref{higherint}.
\begin{proposition}\label{proppath}
	Let $f:\Omega\to \R^m$, where $\Omega$ is a domain in $\R^n$ be an embedding in $W^{1,p}_{loc}(\Omega,\R^m)$, for some $p>n$. Then $f(\Omega)$ equipped with the euclidean metric is a rectifiably connected metric space. In fact for any two points $f(x),f(y)\in f(\Omega)$, $x,y\in \Omega$ there exists a path $\gamma$  connecting $x$ and $y$ for which $d(f(x),f(y))=\ell(f(\gamma))$, where $d$ is the intrinsic metric.
\end{proposition}
\begin{proof}
	For simplicity we will assume that $n=2$ and $m=3$. The argument works similarly for all $n$ and $m$. Let $f(x)$ and $f(y)$ be two points in $f(\Omega)$, $x,y\in\Omega$ and without loss of generality assume that $x=(-1,0)$ and $y=(1,0)$. Consider the path family of straight line segments $\gamma_r(t)=(t,rt+r)$,$t\in[-1,0]$ starting at $x$ and landing at the $y$-axis for $r\in[-1,1]$. We will show that for almost all $r\in[-1,1]$ we  have that 
	\begin{equation}\label{eq0001}
		\int_{-1}^0 |D((f\circ \gamma_r)(t))|dt<\infty.
	\end{equation}
 Indeed notice that \[\int_{-1}^1\int_{-1}^0 |D((f\circ \gamma_r)(t))|^p(t+1)dtdr=\int_{-1}^0\int_{-(x+1)}^{(x+1)}|Df(x,y)(v)|^pdydx,\] where $v=(1,y/x-1)$. The last integral is finite  since $f\in W^{1,p}_{loc}(\Omega,\R^m)$. Thus 	$$\int_{-1}^0 |D((f\circ \gamma_r)(t))|^p(t+1)dt<\infty$$ for almost all $r\in[-1,1]$. By H\"older's inequality and since $p>2$ we have \[\int_{-1}^0 |D((f\circ \gamma_r)(t))|dt\leq C\int_{-1}^0|D((f\circ \gamma_r)(t))|^p(t+1)dt<\infty, \] for almost all $r\in[-1,1]$ as we wanted. Consider now the line segments $\delta_k=\{-1+1/k\}\times [-1+1/k,1-1/k]$, $k\in \N$ and $\delta'=\{0\}\times [-1,1]$. Let $\Gamma$ denote the family of line segments  $\gamma_{r|[-\varepsilon,0]}$ for some $r$ and $\varepsilon>0$ and connect a point of $\delta_k$ to a point in $\delta'$. This path family has positive modulus, $\modue \Gamma>0$. By Fuglede's lemma \ref{fugledelemma} we have that $f$ is absolutely continuous on almost every such path. This implies that the set $E_k\subset  \delta'$ of endpoints of these paths where $f$ is absolutely continuous is of full Lebesgue measure in $\delta'$, see for example \cite[Theorem 5.4]{Hajlasz2003}. Thus the set $E=\bigcap_k E_k$ is of full measure in $\delta'$ and $f$ is absolutely continuous on $\gamma_r$ for almost all $r\in[-1,1]$. Symmetrically, we can prove the same about absolute continuity and equation \eqref{eq0001} if we replace the point $x$ by $y$. Thus for almost all $r\in[-1,1]$ we can find a piecewise linear path $\gamma_r'$ connecting $x$ and $y$ and passing through $(0,r)$ on which $f$ is absolutely continuous. By \eqref{eq0001} we now have that $$\ell(f\circ \gamma_r')=\int_{-1}^0 |D((f\circ \gamma_r')(t))|dt<\infty,$$ as we wanted. The fact that there exists a path realizing the distance between any two points in $f(\Omega)$ follows from \cite[Theorem 3.9]{Hajlasz2003}.
\end{proof}

 Notice here that we assume $f(\Omega)$ to be a metric measure space with the $n$-Hausdorff measure $\mathcal{H}^n_d$ coming from  the metric $d$. Since $(f(\Omega),d,\mathcal{H}^n_d)$ is a metric measure space we can define the Newton--Sobolev space $N^{1,p}_{d,loc}(\Omega,f(\Omega))$. We are  going to need the fact that the Newton--Sobolev spaces with the intrinsic metric $d$ are the same with the regular Sobolev spaces.

\begin{lemma}\label{lemmasobolev}
	For any function $f:\Omega\to\R^m$, where $\Omega$ is a domain of $\R^n$ and $m\geq1$, we have that $f\in N^{1,p}_{d,loc}(\Omega,f(\Omega))$ if and only if $f\in W^{1,p}_{loc}(\Omega,\R^m)$, $1\leq p<\infty$.
\end{lemma}
\begin{proof}
	Suppose first that $f\in N^{1,p}_{d,loc}(\Omega,f(\Omega))$. Then $d(f,z)\in L^p_{loc}(\Omega)$, for some $z\in f(\Omega)$. Thus, since $|f(x)-z|\leq d(f(x),z)$, we obtain that $f\in L^p_{loc}(\Omega,\R^m)$.  Moreover, suppose that $\rho_{f,d}$ is a $p$-weak upper gradient for $f$ with respect to the measure $\mathcal{H}^n_d$ (meaning that this measure was used in the definition of the modulus). Then we have that \begin{equation}\label{eq003}
		|f(\gamma(a))-f(\gamma(b))|\leq d(f(\gamma(a)),f(\gamma(b)))\leq \int_\gamma\rho_{f,d}
	\end{equation}
for $p$-almost every rectifiable curve $\gamma:[a,b]\to \Omega$. Notice that the $p$-almost every curve here, as we said, is with respect to the measure $\mathcal{H}^n_d$. It is easy to see now that since $\mathcal{H}_e^n(A)\leq \mathcal{H}^n_d(A)$ for all measurable sets $A\subset f(\Omega)$, we have that \eqref{eq003} holds for $p$-almost every curve with respect to the $\mathcal{H}^n_e$. Thus $\rho_{f,d}$ is a weak $p$-upper gradient for $f$ with respect to the $\mathcal{H}^n_e$. Thus $f\in N^{1,p}_{loc}(\Omega,f(\Omega))$ and thus by Lemma \ref{minimalgradient} $f\in W^{1,p}_{loc}(\Omega,\R^m).$

On the other hand, suppose $f\in W^{1,p}_{loc}(\Omega,\R^m)$.   We want to show that $$d(f(x),f(z))\in L^p_{loc}(\Omega),\ \text{for some} \ z\in \R^n.$$  Since smooth functions are dense in $W^{1,p}_{loc}(\Omega,\R^m)$ by a standard approximation argument we may assume that $f\in C^\infty(\Omega,\R^m)$. For any $x\in \Omega$, let $\gamma:[a,b]\to \Omega$, for some fixed $a,b\in \R$ be a smooth path with $\gamma(a)=x$ and $\gamma(b)=z$. Notice that $f$ is absolutely continuous on $\gamma$. Then for any compact set $K\subset\Omega$ we have that 
\begin{equation}
	\int_K d(f(x),f(z))^ndx\leq \int_K \ell(f(\gamma))^ndx=\int_K\left(\int_a^b|(f\circ\gamma)'(t)|dt\right)^ndx.
\end{equation} 
Using now Jensen's inequality we obtain \begin{equation}
	\begin{aligned}
	\int_K\left(\int_a^b|(f\circ\gamma)'(t)|\right)^ndtdx&\leq \int_K\int_a^b|Df(\gamma(t))(\gamma'(t))|^ndtdx
\end{aligned}
\end{equation}
which is finite since $Df$ is smooth and thus bounded on any compact set. Thus $d(f(x),f(z))\in L^p_{loc}(\Omega)$. Moreover, let $\rho$ be an upper gradient for $f$ with respect to the euclidean metric. Then by definition 
\begin{equation}\label{eq004}
	|f(\gamma(a))-f(\gamma(b))|\leq \int_\gamma\rho
\end{equation}
for every rectifiable path $\gamma:[a,b]\to \Omega$. Notice now that for every such path we have \begin{equation}\label{eq005}
	d(f(\gamma(a)),f(\gamma(b)))\leq \ell(\gamma)=\sup_i \sum_{i=1}^n |f(x_{i+1})-f(x_i)|
\end{equation}
where the supremum is taken over all partitions $\{x_i\}_{i=1,\dots,n}$ of the path $\gamma$. Suppose now that $\gamma_i$ are the parts of $\gamma$ connecting $x_i$ and $x_{i+1}$. Then equations \eqref{eq004} and \eqref{eq005} imply that \begin{equation}
	d(f(\gamma(a)),f(\gamma(b)))\leq \sup_i \sum_{i=1}^n \int_{\gamma_i}\rho=\int_\gamma \rho.
\end{equation}
Thus $\rho$ is an upper gradient for $f$ and thus a weak $p$-upper gradient with respect to the intrinsic metric $d$. 
	\end{proof}

To prove Theorem \ref{thm5} we need to define the notion of metric differentiability, first introduced by Kirchheim in \cite{Kirchheim1994}. 
\begin{definition}
	Let $f:E\to (X,d_X)$ mapping, where $E\subset \R^n$ measurable and $(X,d_X)$ a metric space. We say that $f$ is metrically differentiable at a point $x\in E$ if there exists a seminorm $\md(f,x)(\cdot)$ on $\R^n$ such that \[d_X(f(z),f(w))-\md(f,x)(z-w)=o(|z-x|+|y-x|),\] when $z,y\to x$, $z,y\in E$.
\end{definition}
One easily sees that if such a seminorm $\md(f,x)(u)$ exists then  $$|\md(f,x)(u)|=\lim_{r\to0}\frac{ d_X(f(x+ru),f(x))}{r}.$$
The next theorem is part of \cite[Corollary 3.4]{Karmanova2007}. We note that $\omega_n$ denotes the volume of the $n$-ball.

\begin{theorem}\label{lemmakarmanova}
	Let $f:\Omega\to(X,d_X)$, where $\Omega$ domain in $\R^n$ and $(X,d_X)$ a metric space, be a function in $N^{1,p}_{loc}(\Omega,X)$, $p>n$. Then $f$ is metrically differentiable almost everywhere and for every measurable set $E\subset \R^n$ we have that
	\begin{equation}
		\int_E J(\md (f,x))dx=\int_X N(f,y,E)d\mathcal{H}^n_{d_X}(y),
	\end{equation}
where $$J(\md (f,x))=\omega_n n\left(\int_{S^{n-1}}\md (f,x)(v)^{-n}d\mathcal{H}^{n-1}(v)\right)^{-1}$$ denotes the Jacobian of the seminorm $\md (f,x)$ and \[N(f,y,E)=\# (f^{-1}(y)\cap E).\] 
\end{theorem}
\subsection{Quasiconformal and quasisymmetric maps between metric spaces}
Here we collect results and definitions from the theory of quasiconformal and quasisymmetric mappings between metric spaces that we are going to need in our study.

We first define the notion of a locally linearly connected metric space. If $X$ is metric space, by $B(x,r)$ and $\overline{B}(x,r)$ we denote an open ball and closed ball respectively, centred at $x\in X$ of radius $r>0$. 
\begin{definition}[Local linear connectivity]
	A metric space $X$ is said to be \emph{linearly locally connected} if there exists a constant $c\geq1$ such that for each $x\in X$ and $r>0$ the following two conditions hold:
	\begin{enumerate}
		\item for any pair of points in $B(x,r)$ there exists a continuum in $B(x,cr)$ that contains both points
		\item for any pair of points in $X\setminus \overline{B}(x,r)$ there exists a continuum in $X\setminus \overline{B}(x,r/c)$ that contains both points
	\end{enumerate}
\end{definition}
We also require the notions of a metric space with locally bounded $n$-geometry and that of an Ahlfors $n$-regular metric space.

\begin{definition}[Locally bounded $n$-geometry]
	A metric measure space $(X,d_X,\mu)$ has \emph{locally bounded $n$-geometry} if it is separable, pathwise connected, locally compact and if there exist constants $C\geq 1$ and $0\leq\lambda\leq1$ and a decreasing function $\psi:(0,\infty)\to(0,\infty)$ such that the following holds: each point in $X$ has an open neighbourhood $U$, compactly contained in $X$, such that \[\mu(B(y,r))\leq C r^n,\] for all balls $B(y,r)\subset U$ and that \[\modu(\Delta(E,F;B(y,r)))\geq \psi(t),\] whenever $B(y,r)\subset U$ and $E,F\subset B(y,\lambda r)$ disjoint, non-degenerate continua with $\dist(E,F)\leq t\min\{\diam E,\diam F\}$. Here $\Delta(E,F;B(y,r))$ denotes the family of paths connecting $E$ to $F$ through $B(y,r)$.
\end{definition} 

\begin{definition}[Ahlfors $n$-regular]
	A metric space $(X,d_X)$ is called \emph{Ahlfors $n$-regular}, $n>0$, if there exists a constant $C\geq 1$ such that \[r^n/C\leq \mathcal{H}^n(B(x,r))\leq Cr^n,\] for all balls $B(x,r)$ in $X$ with $0<r<\diam X$. It is called upper (respectively lower) Ahlfors $n$-regular if the right (respectively left) hand side inequality holds.
\end{definition}
Next we define the notion of quasisymmetry.

\begin{definition}
	A homeomorphism $f:X\to Y$ between metric spaces is said to be \emph{$\eta$-quasisymmetric} for some homeomorphism $\eta:[0,\infty)\to[0,\infty)$ if \[\frac{d_Y(f(x),f(y))}{d_Y(f(x),f(z))}\leq \eta\left(\frac{d_X(x,y)}{d_X(x,z)}\right)\] for all distinct points $x,y,z\in X$. A homeomorphism $f :X \to Y$
	between metric spaces is called \emph{locally $\eta$-quasisymmetric} if each point in $X$ has a neighborhood in which $f$ is $\eta$-quasisymmetric.
\end{definition}
The next theorem that we need is a result due to Williams.

\begin{theorem}[Theorem 1.1 \cite{Williams2012}]\label{williams}
	Let $Q>1$ and $(X,\mu)$, $(Y,\nu)$ be separable, locally finite measure metric spaces and let $f:X\to Y$ be a homeomorphism. Then the following conditions are equivalent:\begin{enumerate}
		\item $f\in N^{1,Q}_{loc}(X,Y)$ and for $\mu$ almost every $x\in X$,\[\rho_f(x)^Q\leq K \mu_f(x).\]
		\item For every family $\Gamma$ of paths in $X$, \[\modu{}_Q(\Gamma)\leq K\modu{}_Q(f(\Gamma)).\]
	\end{enumerate}
\end{theorem} 
Another result that we are going to use is a theorem due to Balogh, Koskela and Rogovin, \cite[Theorem 5.1]{Balogh2007} which tells us when a metrically quasiconformal map belongs to the right Sobolev class.
\begin{theorem}[Theorem 5.1, \cite{Balogh2007}]\label{balogh}
		Let $f:X\to Y$ be a homeomorphism between two locally Ahlfors $n$-regular metric spaces. Assume that $X$ is proper and supports a $1$-Poincare inequality. Suppose that there exists constant $H<\infty$ such that \[H_f(x):=\liminf_{r\to 0}\frac{L_f(x,r)}{\ell(x,r)}<H\]  for Lebesgue almost every $x\in X$. Then $f\in N^{1,n}_{loc}(X,Y)$.
\end{theorem}
\subsection{Properties of quasiconformal curves}
Here we describe the properties of quasiconformal curves that we are going to need. We start with the higher integrability.
\begin{theorem}[Theorem 1.3, \cite{Onninen2021}]\label{higherint}
	Let $f:\Omega\to \R^m$ be a quasiconformal $\omega$-curve for a constant coefficient $n$-volume form in $\R^m$. Then there exists some $p>n$ such that $f\in W^{1,p}_{loc}(\Omega, \R^m)$. 	
\end{theorem}
This theorem has some corollaries which will be important for what follows.
\begin{corollary}\label{lemmadif}
		Let $f:\Omega\to \R^m$ be a quasiconformal $\omega$-curve for a constant coefficient $n$-volume form in $\R^m$. Then $f$ is almost everywhere differentiable. 
\end{corollary}
\begin{corollary}
		Let $f:\Omega\to \R^m$ be a quasiconformal $\omega$-curve for a constant coefficient $n$-volume form in $\R^m$. Then $f$ satisfies Lusin's (N) condition, meaning that $f$ sends sets of zero measure to sets of zero $n$-Hausdorff measure.
\end{corollary}
This immediately implies that a quasiconformal curve for a constant coefficient form satisfies the area formula, see for example \cite{Hajlasz2000} or \cite{Evans1992}.

\begin{theorem}\label{lemmaarea}
		Let $f:\Omega\to \R^m$ be a function in $W^{1,p}_{loc}(\Omega,\R^m)$ for some $p\geq1$ that satisfies Lusin's (N) condition. Then
		\begin{equation}
			\int_E h(x)\sqrt{\det (Df^T(x)Df(x))}dx=\int_{f(E)}\sum_{x\in f^{-1}(y)}h(x)d\mathcal{H}^n(y),
		\end{equation}
		holds for all measurable functions $h:\R^n\to [0,\infty]$ and all measurable sets $E\subset\Omega$.
\end{theorem}

\section{Proof of the first half of Theorems \ref{thmmain}, \ref{thmmetric} and Theorem \ref{thmquasisym}}\label{section3}
We start by constructing a quasiconformal curve whose image is not upper Ahlfors regular with respect to the Hausdorff measure.

\begin{theorem}\label{thmexample}
	There exists a quasiconformal $\omega$ curve $f:\R^2\to\R^3$, where $\omega=dx_1\wedge dx_2$, such that $f(\R^2)$ is not upper Ahlfors regular for the Hausdorff 2-measure. Moreover $f(\R^2)$ is not  linearly locally connected.
\end{theorem}

Since local linear connectivity is preserved under locally quasisymmetric maps and since $\R^2$ is locally linearly connected, the above example shows that quasiconformal curves are not necessarily locally quasisymmetric as stated in Theorem \ref{thmquasisym}. 
\begin{proof}[Proof of Theorem \ref{thmexample}]
	We define the function 
	\begin{equation}
		f(x_1,x_2)=(e^{x_1}\cos x_2,e^{x_1}\sin x_2,\phi(x_2)e^{x_1})		
	\end{equation} 
	where $\phi:\R\to (0,\infty)$ is a smooth, bounded and increasing function with bounded derivative and $\lim_{x\to-\infty}\phi(x)= 0$.
	An easy calculation shows that 
	\begin{equation}
		\star f^*\omega= e^{2x_1}\ \text{and}\ |Df(x_1,x_2)|^2\leq 2e^{2x_1}+e^{2x_1}(\phi(x_2)+\phi'(x_2))^2\leq C e^{2x_1},
	\end{equation}
	where $C\geq1$ is some constant. It is also easy to see that $||\omega||=1$. Moreover, $f:\R^2\to f(\R^2)$, is easily seen to be continuous and injective with a continuous inverse. Thus $f$ satisfies equation \eqref{eq03}. Since $f\in W^{1,2}_{loc}(\R^2,\R^3)$ we have that $f$ is a quasiconformal $\omega$-curve. 
	
	To show that $f(\R^2)$ is not upper Ahlfors regular, choose the point $f(1,-2\pi)=(e,0,\phi(-2\pi)e)$. Assume now, without loss of generality, that $\phi(-2\pi)e<1/10$ and choose a ball in $\R^3$, around $f(1,-2\pi)$ of radius $2$, $B(f(1,-2\pi),2)$. Let \[A=B(f(1,-2\pi),2)\cap f(\Omega).\] $B(f(1,-2\pi),2)$ intersects the plane $\{(x_1,x_2,0):x_1,x_2\in \R\}$ in the set $$D=\{(x_1,x_2,x_3):(x_1-e)^2+x_2^2<r_0,x_3=0\},$$ where $r_0=4-(\phi(-2\pi)e)^2$. We will show that $\mathcal{H}^2(A)=\infty$ so $f(\Omega)$ is not upper Ahlfors regular. Indeed let \[C=\{(x_1,x_2,x_3):(x_1-e)^2+x_2^2<r_0,x_3\in(0,\phi(-2\pi)(e+\sqrt{r_0}))\}.\] Then notice that $C\subset B(f(1,-2\pi),2)$. Also notice that in each strip $$\{(x,y):y\in(-2k\pi,-2(k+1)\pi)\},\  k=1,2,\dots$$ there exists a domain $U_k$ such that $E(U_k)=V$, where $E(z)=e^z$ is the exponential map in the plane and $V=\{(x,y):(x-e)^2+y^2<r_0\}.$ It is easy to see now that $f(U_k)$ is a connected component of the set $C\cap f(\Omega)$. Thus we have \begin{equation}
		\mathcal{H}^2(A)=\int_{B(f(1,-2\pi),2)\cap f(\Omega)}d\mathcal{H}^2\geq \int_{C\cap f(\Omega)}d\mathcal{H}^2=\sum_k\int_{f(U_k)}d\mathcal{H}^2,
	\end{equation}
	 Let $p:\R^3\to \R^3$ be the projection to the $x_1x_2$ plane, i.e. $p(x_1,x_2,x_3)=(x_1,x_2,0)$ and notice that $p(f(U_k))=D$. Since $p$ is a $1$-Lipschitz map we have that $\mathcal{H}^2(f(U_k))\geq \mathcal{H}^2(D)$. Thus \[\mathcal{H}^2(A)\geq \sum_k \mathcal{H}^2(D)\] which implies that $\mathcal{H}^2(A)=\infty$ as we wanted. To show now that $f(\R^2)$ is not locally linearly connected, let $x_n=(10, -2\pi n)$, $n=1,2,\dots$ and consider the balls $$B_n:=B(f(x_n),\phi(-2\pi n)e^{10})$$ in $\R^3$. It is easy to see that  $f(x_{n+1})\in B(f(x_n),\phi(-2\pi n)e^{10}) $ and that for large $n$ any continuum containing $f(x_n)$ and $f(x_{n+1})$ must have points outside of the same ball. In fact it will contain points with  negative $x_1$ coordinate. Since the radius of the balls $B_n$ decreases to zero there is no constant $c\geq 1$ such that we can connect $f(x_n)$ and $f(x_{n+1})$ with a continuum inside $B(f(x_n),c\phi(-2\pi n)e^{10})$ for all $n$ which proves what we wanted.
\end{proof}

Next we restate and prove the first half of Theorem \ref{thmmain}.
\begin{theorem}\label{thm2}
	Let $f:\Omega\to \R^m$ be a quasiconformal $\omega$-curve for a constant coefficient $n$-volume form $\omega$ in $\R^m$. Suppose that  $f(\Omega)\subset\R^m$ is equipped with the induced euclidean metric. Then $f$ is in $N_{loc}^{1,n}(\Omega,\R^m)$ and satisfies the inequality 
	\begin{equation}
		\lip f(x)^n\leq K \mu_f,
	\end{equation} for some constant $K\geq 1$. Moreover, for any path family $\Gamma$ in $\Omega$ we have that  
	\begin{equation}
		\modu\Gamma \leq K \modu f(\Gamma).
	\end{equation}
\end{theorem}
Before we proceed with the proof let us introduce some terminology that we will need. An \emph{$n$-multi-index} is an $n$-tuple, $I=(i_1,\dots,i_n)$ of indices, $1\leq i_1\leq\dots\leq i_n\leq m$. Let $\pi_I:\R^m\to\R^n$ denote the function $\pi_I(x_1,\dots,x_m)=(x_{i_1},\dots,x_{i_n})$ and let $(f_{i_1},\dots,f_{i_n})=f_I=\pi_I\circ f$. We will denote by $J_I$ the Jacobian of the map $f_I:\R^n\to \R^n$.

\begin{proof}
	Since $f\in W^{1,n}_{loc}(\Omega,\R^m)$ we have, by Lemma \ref{minimalgradient}, that $|Df(x)|$ is a minimal $n$-weak upper gradient for $f$  and that $f\in N^{1,n}_{loc}(\Omega,\R^m)$. By Lemma \ref{lemmadif} we have that $\lip f(x)= |Df(x)|$ almost everywhere and by Lemma \ref{lemmaarea} we have that $\mu_f(x)=\sqrt{\det(Df^T(x)Df(x))}$. Since $J_I\leq\sqrt{\sum_IJ_I^2}$, where $J_I$ denotes the $I=(i_1,\dots,i_n)$ minor of $Df$, there always exist a constant $C>0$ such that \[\star f^*\omega \leq C\sqrt{\sum_IJ_I^2},\] where the sum is over all multi-indices $I$. The Cauchy-Binet formula now implies that \[ \sqrt{\det(Df^T(x)Df(x))}=\sqrt{\sum_IJ_I^2},\] where again the sum is over all multi-indices. Thus \[\lip f(x)^n=|Df(x)|^n\leq K\star f^*\omega \leq KC\sqrt{\sum_IJ_I^2}=KC\mu_f, \] as we wanted.  
	For the lower half of the modulus inequality now all we have to show to use \cite[Theorem 1.1]{Williams2012} is that $f(\Omega)$ is a locally finite metric measure space. But this immediately follows from the area formula, Lemma \ref{lemmaarea}, and the fact that $f\in W^{1,n}_{loc}(\Omega,\R^m)$.
\end{proof}
Next we are going to prove the first half of Theorem \ref{thmmetric}.
\begin{theorem}\label{thm1}
	Let $f:\Omega\to \R^m$ be a quasiconformal $\omega$-curve for a constant coefficient $n$-volume form $\omega$ in $\R^m$.  Suppose that $\max_I N(f_I)<\infty$, where the maximum is taken over all multi indices. Then $f:\Omega\to f(\Omega)$ is quasiconformal according to the metric definition.
\end{theorem}
To prove Theorem \ref{thm1} we will first need to show that under the assumptions of the theorem we have that $f(\Omega)$ is upper Ahlfors $n$-regular. The example we constructed in Theorem \ref{thmexample} shows that this is not true without  the assumption that $\max_I N(f_I)<\infty$.
\begin{lemma}\label{lemma1}
	Let $f:\Omega\to \R^m$ be a quasiconformal $\omega$-curve for a constant coefficient $n$-volume form $\omega$ in $\R^m$.  Suppose that $\max_I N(f_I)<\infty$, where the maximum is taken over all multi-indices. Then there exists $C=C(\omega,K)>0$ such that \[\mathcal{H}^n(B(x,r)\cap f(\Omega))\leq C\max_I N(f_I)r^n,\] for all $x\in f(\Omega)$ and all $r>0$.
\end{lemma}

\begin{proof}
	Suppose that $f(U)=B(x,r)\cap f(\Omega)$, where $U$ is an open subset of $\R^n$.	From the area formula, Lemma \ref{lemmaarea},  and Hadamard's inequality we have that \begin{equation}
		\begin{aligned}
			\mathcal{H}^n(f(U))=\int_U \sqrt{\det (Df^T(x)Df(x))}dx\leq \int_U |Df(x)|^ndx.
		\end{aligned}
	\end{equation}
Using now the fact that $f$ is a quasiconformal curve we have that \[\int_U |Df(x)|^ndx\leq K\int_U f^*\omega=K\sum_Ic_I\int_UJ_Idx_I,\] where we have assumed that $\omega=\sum_I c_Idx_I$ and $I$ ranges over all multi-indices.
Thus if we denote by $p_I$ the projection to the coordinates defined by the multi index $I$ and since the functions $f_I$ are in $W^{1,p}_{loc}(\Omega,\R^n)$ using Theorem \ref{lemmaarea} we have that 
\begin{equation}
	\begin{aligned}
		\mathcal{H}^n(f(U))\leq& K\sum_Ic_I\int_UJ_Idx_I\\=&K\sum_Ic_I \int_{f_I(U)} N(f_I,y,U)d\mathcal{H}^n(y)\\\leq& K\max_I N(f_I)\sum_Ic_I\mathcal{H}^n(p_I(B(x,r)\cap f(\Omega)))\\\leq& K\max_I N(f_I)\sum_I c_Icr^n=Cr^n,
	\end{aligned}
\end{equation}
where $c>0$ some constant, as we wanted.
	\end{proof}

Having lemma \ref{lemma1} at our disposal the proof of Theorem \ref{thm1} follows by standard arguments.

\begin{proof}[Proof of Theorem \ref{thm1}]
	Let $x\in f(\Omega)$ and $B_e(x,r):=B(x,r)\cap f(\Omega)$ be a ball around $x$. Let $U(f^{-1}(x),r)=f^{-1}(B_e(x,r))$. For any $0<s<t$ now take $\Gamma(s,t)$ to be the path family in $\Omega$ connecting $\partial U(f^{-1}(x),t)$ and $\partial U(f^{-1}(x),s)$. Let now $L=L(x,r)=\sup_{|x-y|\leq r}|f(x)-f(y)|$ and $\ell=\ell(x,r)=\inf_{|x-y|\geq r}|f(x)-f(y)|$ for any $r>0$. Then by Theorem \ref{thm2} we have that \[\modu(\Gamma(\ell,L))\leq c\modu (f(\Gamma(\ell,L))).\] Moreover, $f(\Gamma(\ell,L))\subset \Delta(S^{m-1}(x,\ell),S^{m-1}(x,L))$, where $\Delta(S^{m-1}(x,\ell),S^{m-1}(x,L))$ is the path family  in $f(\Omega)$ connecting the  euclidean $(m-1)$-spheres $S^{m-1}(x,\ell)$ and $S^{m-1}(x,L)$. From  \cite[Lemma 7.18]{Heinonen2001a} and Lemma \ref{lemma1} we now obtain that \[\modu(f(\Gamma(\ell,L)))\leq C'\left(\log \frac{L}{\ell}\right)^{1-n}.\] Notice now that $U(f^{-1}(x),L)$ contains a ball $B(x,r)$  and $U(f^{-1}(x),\ell)$ has a component $E$ which is contained in the same ball $B(x,r)$. By considering the connected component $A$ of $U(f^{-1}(x),L)$ containing $B(x,r)$ and the connected component $E$ of $U(f^{-1}(x),\ell)$ that is contained in $B(x,r)$ we obtain a ring domain $R(\overline{E},A^c)$, using the terminology in \cite[Chapter 11]{Vaeisaelae1971}. By using  \cite[Theorem 11.7]{Vaeisaelae1971}, we have that the modulus of the path family connecting the boundaries of $E$ and $A^c$ is at least $a_n$ for some constant $a_n$ that depends on $n$. Thus for all small enough $r$, $\modu(\Gamma(\ell,L))>a_n$. Hence we obtain $L(x,r)/\ell(x,r)<C $, for some constant $C>0$ and all small $r>0$. Hence, $H(x,f)<C$ as we wanted.
	\end{proof}

\section{Proof of the second half of Theorems \ref{thmmain} and \ref{thmmetric}}\label{section4}
In this section we state and prove the second half of Theorems \ref{thmmain} and \ref{thmmetric}. We start with Theorem \ref{thmmetric}.
\begin{theorem}\label{thm3}
	Let $f:\Omega\to f(\Omega)\subset \R^m$ be an embedding with $f(\Omega)$ being an Ahlfors $n$-regular metric space with the euclidean metric. If there exist constants $H,C>0$ and an $n$-volume form with constant coefficients $\omega$  such that $\mu_f\leq C \star f^*\omega$ and $H_f(x)\leq H$ for all $x\in \Omega$ then $f$ is a quasiconformal $\omega$-curve.
\end{theorem}
\begin{proof}
By Theorem \ref{balogh} we immediately obtain that $f\in W^{1,n}_{loc}(\Omega,\R^m)$. It remains to prove inequality \eqref{eq03}. To that end, we first need to show that $f$ is differentiable almost everywhere. Indeed, notice that $H(x,f)<H$ almost everywhere, implies that $L(x,r)^n\leq H^n\ell(x,r)^n\leq C \mathcal{H}^n(f(B(x,r)))$, for some constant $C$ that depends only on $n$. Thus \[\left(\frac{L(x,r)}{r}\right)^n\leq C'\frac{\mathcal{H}^n(f(B(x,r)))}{\mathcal{H}^n(B(x,r))}.\] This implies that \[\left(\limsup_{y\to x}\frac{|f(y)-f(x)|}{|y-x|}\right)^n<C' \mu_f,\] for almost every $x\in \Omega$. Since the volume derivative of $f$ is an $L^1_{loc}$ function this implies that $\limsup_{y\to x}\frac{|f(y)-f(x)|}{|y-x|}<\infty$ for almost every $x\in \Omega$. The Rademacher--Stepanoff theorem, see \cite[Theorem 6.1.1]{Iwaniec2001}, implies that $f$ is differentiable almost everywhere. On points where $f$ is differentiable now we have that \[\limsup_{r\to 0}\frac{\sup_{|x-y|\leq r}|f(x)-f(y)|}{r}\leq H\limsup_{r\to 0}\frac{\inf_{|x-y|\geq r}|f(x)-f(y)|}{r}.\] This implies that \[|Df(x)|^n\leq K \limsup_{r\to0}\frac{\mathcal{H}^n(f(B(x,r)))}{\mathcal{H}^n(B(x,r))}=K\mu_f\leq KC \star f^*\omega.\] Thus $f$ is a quasiconformal $\omega$-curve.
	\end{proof}
Next we prove the second half of Theorem \ref{thmmain}.
\begin{theorem}\label{thm4}
	Let $f:\Omega\to f(\Omega)\subset \R^m$ be an embedding with $f(\Omega)$ being  locally finite with the Hausdorff $n$-measure of $\R^m$. If there exist constants $C>0$, $K\geq 1$ and an $n$-volume form with constant coefficients $\omega$  such that $\mu_f\leq C \star f^*\omega$ and $f$ either
	\begin{enumerate}
		\item is in $N^{1,n}_{loc}(\Omega,\R^m)$ and $\lip f(x)^n\leq K \mu_f$ for almost every $x\in \Omega$ or
		\item $\modu \Gamma\leq K\modu f(\Gamma)$, for any path family in $\Omega$,
	\end{enumerate}  then $f$ is a quasiconformal $\omega$-curve.
\end{theorem}
\begin{proof}
	From Williams \cite[Theorem 1.1]{Williams2012} we have that under our assumptions conditions (1) and (2) are equivalent. We already know that $f\in N^{1,n}_{loc}(\Omega,\R^m)$ if and only if $f\in W^{1,n}_{loc}(\Omega,\R^m)$. Also  it is easy to see that, for   $f\in W^{1,n}_{loc}(\Omega,\R^m)$, $\lip f$ is an $n$-weak upper gradient.  Indeed, by Lemma~\ref{fugledelemma} $f$ is absolutely continuous on $n$-almost every path. Thus for $n$-almost every such path $\gamma:[a,b]\to \Omega$ we have that \begin{equation*}
		\begin{aligned}|f(\gamma(a))-f(\gamma(b))|&= \int_0^{\ell(\gamma)}|(f(\gamma(t)))'|dt\\&=\int_0^{\ell(\gamma)}|Df(\gamma(t))(\gamma'(t))|dt\\&\leq \int_0^{\ell(\gamma)}\lip f(\gamma(t))dt=\int_\gamma\lip f ds.\end{aligned}\end{equation*}
	
	 We know now from Lemma \ref{minimalgradient} that $|Df(x)|$ is the minimal upper gradient  and thus $|Df(x)|\leq \lip f(x)$. Hence\[|Df(x)|^n\leq \lip f(x)^n\leq K \mu_f\leq K C \star f^*\omega.\] Thus $f$ is a quasiconformal $\omega$-curve.
	\end{proof}
\section{Proof of Theorems \ref{thm5} and \ref{thmpath}}\label{section5}
First we are going to prove that any embedding $f:\Omega\to\R^m$ in $W^{1,p}_{loc}(\Omega, \R^m)$, $p>n$ is a homeomorphism when looked as a map from $\Omega$ to $f(\Omega)$ and $f(\Omega)$ is equipped with its intrinsic metric.
\begin{lemma}\label{lemmahomeo}
	Let $f:\Omega\to\R^m$ be an embedding in $W^{1,p}_{loc}(\Omega,\R^m)$ for some $p>n$. Suppose that $f(\Omega)$ is equipped with its intrinsic metric. Then $f:(\Omega,e)\to (f(\Omega),d)$ is a homeomorphism.
\end{lemma}
\begin{proof}
	It is easy to show that $f^{-1}:(f(\Omega),d)\to (\Omega,e)$ is continuous and $f$ is obviously injective. So we only have to prove that $f:(\Omega,e)\to (f(\Omega),d)$ is continuous as well. To that end, let $x\in \Omega$ and $B(x,\delta)$ ball of radius $\delta>0$. We will show that for all $\varepsilon,\varepsilon'>0$ the points in $B(f(x),\varepsilon)\cap f(\Omega)$ are at a distance at most $\varepsilon'$ with the intrinsic metric when $\delta$ is small enough. Indeed, suppose that there exists a $\varepsilon'>0$ such that for all $\varepsilon$ there existed point$y\in B(x,\delta)$ and  $f(y)\in B(f(x),\varepsilon)\cap f(\Omega)$ such that $d(f(x),f(y))>\varepsilon'$ for all $\delta>0$. This implies that for all paths $\gamma$ in $f(\Omega)$ connecting $f(x)$ to $f(y)$ we have that $\ell(\gamma)>\varepsilon'$. By arguing as  in Proposition \ref{proppath} we can find a path family $\gamma_r$ connecting $x$ and $y$ on which $f$ is absolutely continuous. For every path $\gamma_r$ in this path family we will have that  $$\ell(f\circ \gamma_r)=\int_a^b |D((f\circ\gamma_r)(t))| dt\geq \varepsilon',$$ for some $a,b$ which tend to the same number as $\delta$ gets smaller. Thus \[\int_{-1}^1\int_a^b |D((f\circ\gamma_r)(t))|^p (t+1)dtdr\geq2\varepsilon',\] for all $\delta>0$. By a change of variables we see now that this integral is greater than the integral of$|Df(x)|$ over some domain that shrinks to a zero measure set as $\delta$ gets smaller but this is impossible since $|Df(x)|$ is in $L^p_{loc}(\Omega)$. Thus for all $\varepsilon'$ and all $y\in B(x,\delta)$ we have that $d(f(x),f(y))\leq \varepsilon'$ when $\delta$ is small enough and $f$ is continuous.
\end{proof}

To prove Theorem \ref{thm5} we are going to need the following lemma.
\begin{lemma}\label{lemmaradon}
	Let $f:\Omega\to \R^m$ be a quasiconformal $\omega$-curve, where $\Omega$ domain in $\R^n$ and $\omega$ a constant coefficient $n$-volume form in $\R^m$. Then 
	\begin{enumerate}
		\item $f$ is almost everywhere metrically differentiable and, for almost every $x\in \Omega$, $r>0$ and $\mathcal{H}^{n-1}$-almost every $v\in S^{n-1}$, we have that
	\begin{equation}
		\md(f,x)(v)=\lim_{r\to0}\frac{d(f(x+rv),f(x))}{r}=|Df(x)(v)|,
	\end{equation}

\item and for almost every $x\in \Omega$ we have that 
\begin{equation}
	\lip{}_d f(x)=\limsup_{y\to x}\frac{d(f(x),f(y))}{|x-y|}= |Df(x)|.
\end{equation}
\end{enumerate}
\end{lemma}
\begin{proof}
	(1) By \ref{lemmasobolev} and Theorem  \ref{higherint} we obtain that $f\in N^{1,p}_{loc}(\Omega,f(\Omega))$ for some $p>n$. Thus by Theorem \ref{lemmakarmanova} we have that $f$ is almost everywhere metrically differentiable. 	To prove \begin{equation}
		\md(f,x)(v)=\lim_{r\to0}\frac{d(f(x+rv),f(x))}{r}=|Df(x)(v)|,
	\end{equation} let $x$ be  a point in $\Omega$  and  $v$ is any vector in $S^{n-1}$. Let $$\gamma_v:[0,r]\to \Omega, \ \gamma_v(t)=x+tv$$ be a straight line segment in $\Omega$ connecting $x$ and  $x+rv$ parametrized by arc length, when $r$ is small enough. By the definition of $d$ we have that 
	\begin{equation}\label{eq611}
	\left| d(f(x+rv),f(x))\right|\leq 	\left| \ell(f\circ \gamma_v)\right|.
	\end{equation}
 By Lemma \ref{fugledelemma} we have that  $f$ is absolutely continuous on $\gamma_{v|[\delta,r]}$, for all $\delta>0$ and almost all $v\in S^{n-1}$. Thus for almost all $v\in S^{n-1}$ we have that \begin{equation}
	\begin{aligned}
	\ell(f\circ \gamma_{v|[\delta,r]})&=\int_\delta^r |(f\circ\gamma_v)'(t)|dt=\int_\delta^r|Df(\gamma_{v|[\delta,r]}(t))(v)|dt.
	\end{aligned}
\end{equation}
Thus we have proved that 
\begin{equation}\label{eq006}
		\left| d(f(x+rv),f(x+\delta v))\right|\leq \int_\delta^r|Df(\gamma_{v|[\delta,r]}(t))(v)|dt,
\end{equation}
holds a.e. in $\Omega$. Taking $\delta\to 0$ we have that 
\begin{equation}\label{eq007}
	\left| d(f(x+rv),f(x))\right|\leq \int_0^r|Df(\gamma_{v}(t))(v)|dt.
\end{equation}
 Since $|Df(x)|$ is in $L^1_{loc}(\Omega)$, we have that for almost all $v\in S^{n-1}$, the function $|Df(\gamma_v(t))|$, $t\in [0,r]$ is integrable. Thus by Lebesgue's differentiation theorem we now obtain
\[\limsup_{r\to0}\frac{d(f(x+rv),f(x))}{r}\leq \limsup_{r\to0}\frac{1}{r}\int_0^r|Df(\gamma_v(t))(v)|dt=|Df(x)(v)|.\]
 On the other hand, we know that 
\[\frac{|f(x+rv)-f(x)|}{r}\leq\frac{d(f(x+rv),f(x))}{r}\] and thus \[|Df(x)(v)|\leq\liminf_{r\to0}\frac{d(f(x+rv),f(x))}{r}.\] Hence, the limit exists and we have equality.

(2) Arguing as in (1) we choose a line segment $\gamma$ connecting $x$ and $y$ and assume that $f$ is smooth. Then we have that 
\begin{equation}
	d(f(y),f(x))\leq \int_0^{\ell(\gamma)}|Df(\gamma(t))(\gamma'(t))|dt\leq \int_0^{\ell(\gamma)}|Df(\gamma(t))|dt.
\end{equation}
Thus
\begin{equation}
	\frac{d(f(y),f(x))}{|x-y|}\leq \frac{1}{|x-y|}\int_0^{|x-y|}|Df(\gamma(t))|dt.
\end{equation}
Again by Lebesgue's differentiation theorem we have that \begin{equation}
	\limsup_{y\to x} \frac{d(f(x),f(y))}{|x-y|}\leq |Df(x)|.
\end{equation}
Since smooth functions are dense in $W^{1,n}_{loc}(\Omega,\R^m)$ and $W^{1,n}_{loc}(\Omega,\R^m)=N^{1,n}_{d,loc}(\Omega,f(\Omega))$ we have that the above inequality is true for functions in $W^{1,n}_{loc}(\Omega,\R^m)$ and for almost all $x\in \Omega$. 
On the other hand, since $f$ is differentiable almost everywhere we have that \[|Df(x)|=\limsup_{y\to x} \frac{|f(x)-f(y)|}{|x-y|}\leq \limsup_{y\to x} \frac{d(f(x),f(y))}{|x-y|}\] and thus we have equality.
\end{proof}
\begin{proof}[Proof of Theorem \ref{thm5}]
	Theorem \ref{lemmakarmanova} implies that the Radon-Nikodym derivative of the measure $\nu_1(A)=\mathcal{H}^n_d(f(A))$ is $J(\md(f,x))$. But we know from Lemma \ref{lemmaradon} that $J(\md(f,x))=J(|Df(x)|)$. Now Corollary \ref{lemmaarea} and Theorem \ref{lemmakarmanova} imply that 
	\begin{equation}
		\nu_2(A)=\mathcal{H}^n_e(f(A))=\int_A \sqrt{\det Df(x)^TDf(x)}dx=\int_A J(|Df(x)|)dx.
	\end{equation}
Thus the Radon-Nikodym derivative of the measure $\nu_2$ coincides with that of the measure $\nu_1$. Hence $\mathcal{H}^n_d(f(A))=\mathcal{H}^n_e(f(A))$ as we wanted.
\end{proof}
Next we prove the first half of Theorem \ref{thmpath}.
\begin{theorem}\label{corollary}
	Let $f:\Omega\to \R^m$ be a quasiconformal $\omega$-curve for a constant coefficient $n$-volume form $\omega$ in $\R^m$. Suppose that $f:\Omega\to f(\Omega)$ is a homeomorphism when $f(\Omega)$ is equipped with its intrinsic metric. Then $f$ is in $N_{d,loc}^{1,n}(\Omega,f(\Omega))$ and satisfies the inequality 
	\begin{equation}
		\lip{}_d f(x)^n\leq K \mu_{d,f},
	\end{equation}
	for some $K\geq1$ and almost every $x\in \Omega$. Moreover, there exists a constant $K'\geq1$ such that for any path family $\Gamma$ in $\Omega$ we have that  
	\begin{equation}
		\modu\Gamma \leq K' \modup f(\Gamma).
	\end{equation}
\end{theorem}
\begin{proof}
	By Theorem \ref{thm5} we have that $\mathcal{H}^n_d(A)=\mathcal{H}^n_e(A)$ for any measurable set $A$ in $f(\Omega)$. Thus \[\modu (\Gamma)=\modup (\Gamma),\] for every path family $\Gamma$ in $f(\Omega)$. Thus Theorem \ref{thm2} implies that $f$ satisfies the lower half of the geometric definition \[	\modu\Gamma \leq K' \modup f(\Gamma).\] On the other hand Lemma \ref{lemmasobolev}  implies that  $f$ is in $N_{loc}^{1,n}(\Omega,f(\Omega))$. Moreover, by assumption we have that 
	\begin{equation}
	|Df(x)|^n\leq \frac{K_0}{||\omega||\circ f}\star f^*\omega,
		\end{equation}
	for almost every $x\in \Omega$ and some $K_0\geq1$. By arguing as in the proof of Theorem \ref{thm2} we have that $\star f^*\omega\leq C \mu_f$, for some constant $C>0$ and by Lemma \ref{lemmaradon} we have that $\lip{}_df(x)= |Df(x)|$. Thus \[\lip{}_df(x)^n\leq K \mu_{d,f},\] for some $K\geq1$.
	\end{proof}
Finally, we prove the second half of Theorem \ref{thmpath}.
\begin{theorem}\label{thm6}
	Let $f:\Omega\to f(\Omega)\subset \R^m$ be an embedding. Assume that $f(\Omega)$ has  locally finite $\mathcal{H}_d^n$ measure. If there exist constants $C>0$, $K\geq 1$ and an $n$-volume form with constant coefficients $\omega$  such that $\mu_{d,f}\leq C \star f^*\omega$ and $f$ either
	\begin{enumerate}
		\item is in $N^{1,n}_{d,loc}(\Omega,f(\Omega))$ and $\lip_d f(x)^n\leq K \mu_{d,f}$ for almost every $x\in \Omega$ or
		\item $\modu \Gamma\leq K\modup f(\Gamma)$, for any path family in $\Omega$,
	\end{enumerate}  then $f$ is a quasiconformal $\omega$-curve.
\end{theorem}
\begin{proof}
	From Theorem \ref{williams} and Lemma \ref{lemmahomeo} we have that (1) and (2) are equivalent. From Lemma \ref{lemmasobolev} we have that $f\in W^{1,n}_{loc}(\Omega,\R^m)$. It is easy to see that $\lip f(x)\leq \lip_d f(x)$ and by arguing as in the proof of Theorem \ref{thm4}, $|Df(x)|\leq \lip f(x)$. Assuming that (1) holds now,  we have \[|Df(x)|^n\leq \lip f(x)^n\leq \lip{}_df(x)^n\leq K\mu_{d,f}\leq KC\star f^*\omega.\] Thus $f$ is a quasiconformal $\omega$ curve.
\end{proof}
\bibliographystyle{amsplain}
\bibliography{bibliography}
\end{document}